\newif\ifbw
\bwtrue

\documentclass{amsart}
\usepackage{tikz}

\newtheorem{theorem}{Theorem}[section]
\newtheorem{lemma}[theorem]{Lemma}
\newtheorem{proposition}[theorem]{Proposition}
\newtheorem{conjecture}[theorem]{Conjecture}
\newtheorem{statement}[theorem]{Statement}

\newtheorem{corollary}[theorem]{Corollary}

\newtheorem*{proposition*}{Proposition}
\newtheorem*{conjecture*}{Conjecture}
\newtheorem*{theorem*}{Theorem}

\newcommand{\nesetril}{Ne\v set\v ril}
\newcommand{\kk}{\mathbf{k}}
\renewcommand{\ll}{\mathbf{l}}
\newcommand{\LL}{\mathbf{L}}
\newcommand{\mm}{\mathbf{m}}
\newcommand{\MM}{\mathbf{M}}
\newcommand{\nn}{\mathbf{n}}
\renewcommand{\ss}{\mathbf{s}}

\makeatletter
\@namedef{subjclassname@2020}{\textup{2020} Mathematics Subject Classification}
\makeatother

\DeclareMathOperator{\Proj}{\mathrm{Proj}}

\title{Ramsey properties of products of chains}

\author{Csaba Bir\'o}
\address{Department of Mathematics, University of Louisville, Louisville, KY 40292, USA}
\email{csaba.biro@louisville.edu}

\author{Sida Wan}
\address{Department of Mathematics, University of Louisville, Louisville, KY 40292, USA}
\email{sida.wan@louisville.edu}

\begin{document}

\begin{abstract}
Let $\kk$ denote the totally ordered set (or chain) on $k$ elements. The product $\kk^t=\kk\times\cdots\times\kk$ is a poset called a grid. This paper discusses several loosely related results on the Ramsey theory of grids. Most of the results involve some application of the Product Ramsey Theorem.
\end{abstract}


\maketitle

\section{Motivation}
The original motivation of this paper was two questions of \nesetril\ and Pudl\'ak from 1986. In their paper \cite{NesetrilPudlak}, they introduced the notion of the Boolean dimension of partially ordered sets (see the first paragraphs of Section~\ref{s:np}). They proved an upper and a lower bound based on the number of points of the poset. At the end of their note they asked two questions. They did not voice their thoughts about which way the questions would go; nevertheless we rephrase the questions as ``statements'' for easier discussion.
\begin{statement}\label{sta:np1}
The Boolean dimension of planar posets is unbounded.
\end{statement}

Note that this is the opposite of what many researchers, possibly including \nesetril\ and Pudl\'ak, conjectured. E.g.\ in \cite{local}, the authors state that it is clear from the presentation of the question in \cite{NesetrilPudlak} that they believed the answer should be ``no''.

We will say that a class $\mathcal{C}$ of posets has the Ramsey-property, if for all $r$ and $P\in\mathcal{C}$, there is a poset $Q\in\mathcal{C}$, such that for every $r$-coloring of the comparabilities of $Q$, there is a subposet $Q'$ of $Q$ that is isomorphic to $P$ such that every comparability of $Q'$ is of the same color.

\nesetril\ and R\"odl \cite{NesetrilRodl} proved that the class of all posets has the Ramsey property. (In fact, this is just a consequence of their much stronger theorem.) Later we will show that this special consequence is also a rather direct consequence of the so-called Product Ramsey Theorem.

\nesetril\ and Pudl\'ak asked a second question in their paper, which we also phrase as a statement.

\begin{statement}\label{sta:np2}
The class of planar posets has the Ramsey property.
\end{statement}

\nesetril\ and Pudl\'ak pointed out that Statement~\ref{sta:np2} implies Statement~\ref{sta:np1}, though they did not include the proof in their short article.
We provide a proof of this implication in Section~\ref{s:np}. We do this, after we set up the basic framework of the Ramsey-property for relational sets, and we prove a useful general lemma in Section~\ref{s:relational}.

We take a slight detour in Section~\ref{s:planar} to discuss some related statements for planar graphs.

When we started to work on present research, we guessed both Statement~\ref{sta:np1} and \ref{sta:np2} are false. Since we did not know how to approach the more general question, we decided we would attempt to disprove Statement~\ref{sta:np2}. One natural way to do that would be to construct a planar poset $P$ for which there is not an appropriate $Q$ required by the Ramsey-property. A simple choice for a planar poset would be a 2-dimensional grid. However, we quickly realized that that is not a good example. In fact, we proved the following.

\begin{proposition*}
For each $t$ positive integer, the class of $t$-dimensional grids have the Ramsey Property.
\end{proposition*}

As we will see in Section~\ref{s:grid}, this proposition is a relatively simple consequence of the Product Ramsey Theorem, and easily implies that the class of all posets has the Ramsey-property.

We considered developing a tool that would be more powerful than the Product Ramsey Theorem. The result of this attempt is the following conjecture.

\begin{conjecture*}
For all $t$, $r$, $m$, and $l$ there exists an $n$ such that for all $r$-colorings of the $\mm^t$ subposets of $\nn^t$, there is a monochromatic $\ll^t$ subposet $L$. That is, every $\mm^t$ subposet of $L$ receives the same color.
\end{conjecture*}

In Section~\ref{s:subposet}, we prove the conjecture for the special case $t=2$.

\begin{theorem*}
For all $r$, $m$, and $l$ there exists an $n$ such that for all $r$-colorings of the $\mm^2$ subposets of $\nn^2$, there is a monochromatic $\ll^2$ subposet $P$. That is, every $\mm^2$ subposet of $P$ receives the same color.
\end{theorem*}

In Section~\ref{s:extension} we, again, use the Product Ramsey Theorem to prove the following, somewhat counterintuitive result. This is a generalization of a classical result by Paoli, Trotter, and Walker \cite{PTW}.

\begin{theorem*}
Let $X$ be a poset and let $M$ be a linear extension of $X$. Furthermore, let $k$ be a positive integer. Then there exists a grid $Y\cong \nn^t$ such that for all $L_1,L_2,\ldots,L_k$ linear extensions of $Y$, there is a subposet $X'$ of $Y$ such that
\begin{itemize}
\item $X'\cong X$, evidenced by the embedding $f:X\to Y$;
\item for all $i=1,\ldots,k$, and for all $a,b\in X$, we have $a<b$ in $M$ if and only if $f(a)<f(b)$ in $L_i$.
\end{itemize}
\end{theorem*}

We note that, although some of the problems studied may be interesting for infinite posets, in this paper every poset is finite. In fact, we will omit the word finite, and even when we say, e.g.\ ``class of all posets'', we mean ``class of all finite posets''.

\section{Relational sets}\label{s:relational}

First we define the Ramsey property of general classes of sets with relations to prove a general lemma that shows that the number of colors (as long as it is at least $2$) does not matter.

Let $X$ be a set, and $r$ a positive integer. An \emph{$r$-coloring} of $X$ is a function $c:X\to[r]$, where $[r]=\mathcal\{1,\ldots,r\}$. The elements of $[r]$ are called \emph{colors}. Indeed, any function $g: X\to S$ where $|S|=r$ can be considered an $r$-coloring. A \emph{relation} $R$ on $X$ is a subset $R\subseteq X\times X$. If $X'\subseteq X$, we use the usual notation $c|_{X'}$ and $R|_{X'}$ for the restriction of $c$ and $R$ (respectively) to the subset $X'$.

Let $\mathcal{C}$ be a class of ordered pairs $(X,R)$, where $X$ is a set, and $R$ is a relation on $X$. We say that $\mathcal{C}$ has the \emph{Ramsey Property}, if for all $(X,R)\in\mathcal{C}$ and for all $r$ positive integer, there exists $(Y,S)\in\mathcal{C}$ such that for every $r$-coloring $c$ of $S$, there exists a subset $Y'\subseteq Y$ such that if $S'=S|_{Y'}$, then $(Y',S')\cong (X,R)$, and for all $a,b\in S'$, we have $c|_{S'}(a)=c|_{S'}(b)$.
Less formally $\mathcal{C}$ has the Ramsey property, if for all $\mathbf{X}\in\mathcal{C}$ there is a (larger) $\mathbf{Y}\in\mathcal{C}$, such that if we $r$-color the relations of $\mathbf{Y}$, we will find a monochromatic subrelation $\mathbf{X'}$ of $\mathbf{Y}$ that is isomorphic to $\mathbf{X}$. Monochromatic means that every relation of $\mathbf{X'}$ is assigned the same color. The set (with the relation) $\mathbf{Y}$ is called the \emph{Ramsey set} of $X$.

With a slight abuse of notation, we will often write $X$ for the pair $\mathbf{X}=(X,R)$.

After this definition one might think that the classical theorem of Ramsey could be rephrased by perhaps saying that the set of (complete) graphs have the Ramsey Property, using the usual definition of a graph as an irreflexive, symmetric relation. This is, however, not the case. Clearly, one can $2$-color the relations of a graph by coloring the two directions of an edge opposite colors for each edge, and then no monochromatic edge will even be found.

On the other hand, it \emph{is} possible to state Ramsey's Theorem with this terminology: it is the statement that the class of linear orders has the Ramsey Property. This will be a special case of our Proposition~\ref{gridramsey}.

The following lemma is often useful when one is trying to prove that a class has the Ramsey Property.

\begin{lemma}
Suppose $\mathcal{C}$ has the following property: for all $X\in\mathcal{C}$ there is a $Y\in\mathcal{C}$ and a positive integer $r_0\geq 2$, such that if we $r_0$-color the relations of $Y$, we will find a monochromatic subrelation $X'$ of $Y$ that is isomorphic to $X$. Then $\mathcal{C}$ has the Ramsey Property.
\end{lemma}

\begin{proof}
Let $\mathcal{C}$ be a class, $X\in\mathcal{C}$, and $r$ a positive integer; we need to show that a Ramsey set $Y$ can be found. We will do that by induction on $r$. If $r\leq 2$, then by the conditions there exists $Y$ and $r_0\geq 2$. Since $r\leq r_0$, an $r$-coloring is a special $r_0$-coloring, so the statement follows.

Now let $r>2$, and assume the statement is true for $r-1$. So there exists $Y\in\mathcal{C}$, such that if we $r-1$-color the relations of $Y$, there is a monochromatic subrelation $X'$ of $Y$ that is isomorphic to $X$.

We can use the hypothesis again for $Y\in\mathcal{C}$ and $2$-colors. There exists a $Z\in\mathcal{C}$ such that any 2-coloring of the relations of $Z$ yields a monochromatic copy of $Y$. We claim that $Z$ is a correct choice for the original set $X$ and $r$ colors.

To see this, consider an $r$-coloring $c$ of the relations of $Z$. Now recolor $Z$ with only $2$ colors based on the $c$: if $c(x)=1$, use the color blue; otherwise use the color red. We know $Z$ yields a monochromatic $Y$. If $Y$ is blue, then we notice that $X$ is a subrelation of $Y$, so we found a monochromatic $X$. If $Y$ is red, then we revert to $c$ to color the relations of $Y$ with $r-1$ colors, and we find the monochromatic $X$ this way.
\end{proof}

For the balance of this paper we assume the reader is familiar with basic notions of partially ordered sets and graph theory. We refer the reader to the monograph of Trotter \cite{Trotter}, and the textbook of Diestel \cite{Diestel}.

\section{Two questions of \nesetril\ and Pudl\'ak}\label{s:np}

Recall that the Iverson bracket is a notation that converts a logical proposition to $0$ or $1$: $[P]=1$ if $P$ is true, and $[P]=0$, if $P$ is false.

Let $P$ be a poset, and let $(\mathcal{L},S)$ be a pair where $\mathcal{L}=\{L_1,\ldots,L_d\}$, $(d\geq 1$) is a set of linear orders of the elements of $P$, and $S$ is a set of binary (0--1) strings of length $d$. For two distinct elements $x,y\in P$, let $P_i(x,y)$ be the proposition that $x<y$ in $L_i$. We call $(\mathcal{L},S)$ a \emph{Boolean realizer}, if for any two distinct elements $x,y$, we have $x<y$ in $P$ if and only if $[P_1(x,y)] [P_2(x,y)] \ldots [P_d(x,y)]\in S$. We call this binary string the \emph{signature} of the pair $(x,y)$. The number $d$ is the cardinality or size of the Boolean realizer. The minimum cardinality of a Boolean realizer is the Boolean dimension of $P$, denoted by $\dim_B(P)$.

We note that there are minor variations in the the definition of Boolean realizers in the literature. (See next paragraph for citations.) With our definition, antichains are of Boolean dimension $1$ (one can take $S=\emptyset$), chains are of Boolean dimension $1$, and in general, $\dim_B(P)\leq\dim(P)$, because a Dushnik--Miller realizer $P$ can be easily converted into a Boolean realizer of the same size by taking $S=\{11\ldots 1\}$.

Boolean dimension and structural properties of posets have seen an increased interest in recent years, e.g.\ in \cite{FelsnerMeszarosMicek}, the authors showed that posets with cover graphs of bounded tree-width have bounded Boolean dimension. Further, in \cite{dimensions}, the authors compared the Dushnik--Miller dimension, Boolean dimension and local dimension in terms of tree-width of its cover graph, and in \cite{MeszarosMicekTrotter}, the authors studied the behavior of Boolean dimension with respect to components and blocks.

As mentioned earlier, the following statement appeared without proof in \cite{NesetrilPudlak}. We include a proof for completeness.

\begin{proposition}
Statement~\ref{sta:np2} implies Statement~\ref{sta:np1}.
\end{proposition}

\begin{proof}
Assume that Statement~\ref{sta:np2} is true, but the Boolean dimension of planar posets is at most $k$. Let $P$ be a planar poset whose Dushnik--Miller dimension is greater than $k$ (such a poset is well-known to exist). By Statement~\ref{sta:np2}, there is a planar poset $Q$ such that any $2^k$-coloring of the comparabilities of $Q$ yields a monochromatic $P$.

Let $(\mathcal{L},S)$ be a Boolean realizer of size $k$ of $Q$, and let $\mathcal{L}=\{L_1,\ldots,L_k\}$. Color the comparabilities of $Q$ with binary strings of length $k$ as colors: if $x<y$ in $Q$, let the color of $(x,y)$ be the signature of $(x,y)$.

Now let $P'$ be a subposet of $Q$ such that $P'\cong P$ and every comparable pair of $P'$ is of the same color, say $d_1d_2\ldots d_k$ (where $d_i$ is the $i$th digit of the binary string). Let $M_i=L_i$ if $d_i=1$, and let $M_i=L_i^d$ (the dual of $L_i$), if $d_i=0$. It is routine to verify that $\{M_1,\ldots,M_k\}$ is a realizer of $P$, contradicting $\dim(P)>k$.
\end{proof}

\section{Ramsey property of planar graphs}\label{s:planar}

We noted earlier that our general notion of Ramsey Property is not very natural for studying graphs, because we can color the two directions of an edge with different colors. So it is natural to redefine the Ramsey Property specifically for classes of graphs.

We say that a class of graphs $\mathcal{C}$ has the \emph{Ramsey Property}, if for all $G\in\mathcal{C}$ and for all $r$ positive integers, there exists $H\in\mathcal{C}$ such that for every $r$-coloring of $E(H)$, there exists an induced subgraph $G'$ of $H$ such that $G'\cong G$, and every edge of $G'$ is of the same color. We use the term ``monochromatic'' as before, and we call the graph $H$ the Ramsey graph of $G$.

Ramsey's Theorem can be restated by saying the class of complete graphs has the Ramsey Property. The fact that the class of all graphs has the Ramsey Property is a more difficult statement, and it was proven around 1973 independently by Deuber \cite{Deuber}, by Erd\H os, Hajnal and P\'osa \cite{ErdosHajnalPosa}, and by R\"odl \cite{Rodl-MT}.

We note that as for general relations, the analogous lemma is true and can be proven exactly the same way.

\begin{lemma}
Suppose $\mathcal{C}$, a class of graphs, has the following property: for all $G\in\mathcal{C}$ there is a $H\in\mathcal{C}$ and a positive integer $r_0\geq 2$, such that if we $r_0$-color the edges of $H$, we will find a monochromatic induced subgraph $G'$ of $H$ that is isomorphic to $G$. Then $\mathcal{C}$ has the Ramsey Property.
\end{lemma}

Motivated by our problem on planar posets, we were curious whether the class of planar graphs has the Ramsey Property. This is not the case. In fact, as Axenovich et al.\  \cite{planar} pointed out, a result of Gon\c calves \cite{Goncalves} and the Four Color Theorem imply that if $G$ has an appropriate $H$ as in the definition, then $G$ must be planar bipartite. But to just prove that the class of planar graphs does not have the Ramsey Property, we only need elementary tools.

\begin{proposition}
The class of planar graphs do not have the Ramsey Property.
\end{proposition}

\begin{proof}
Let $G$ be a planar graph that is not bipartite. Now suppose that the class of planar graphs has the Ramsey Property. Then there exists a planar Ramsey graph $H$ for $G$. Since $\chi(H)\leq 6$, one can decompose the edge set of $H$ into $\binom{6}{2}$ bipartite graphs. (These numbers can obviously be improved.) Color the edges of $H$ based on which of these bipartite graphs they are in. Let $G'\cong G$ be a monochromatic induced subgraph of $H$. Since the edges of $G'$ use a single color, $G'$ is bipartite, a contradiction.
\end{proof}

\section{Ramsey property of grids}\label{s:grid}

We will use $\kk$ to denote the $k$-element chain, and $\kk^t$ for the poset that is the product of the $k$-element chain by itself $t$ times. In more details, suppose the ground set of $\kk$ is $X=\{x_1,x_2,\ldots,x_k\}$ with $x_1<x_2<\cdots<x_k$. Then the elements of $\kk^t$ are $t$-tuples $(x_{i_1},x_{i_2},\ldots,x_{i_t})$ with $1\leq i_l\leq k$ for all $l$, and $(x_{i_1},\ldots,x_{i_t})\leq (x_{j_1},\ldots,x_{j_t})$ in $\kk^t$ if and only if $i_l\leq j_l$ for all $l$.

 This poset will be called the $\kk^t$ grid. The number $t$ is called the dimension of the grid. This coincides with the Dushnik--Miller dimension of the poset for $k\geq 2$, so it will not cause confusion.

Let the ground set of the poset $\nn^t$ be the set of $t$-tuples of numbers from $[n]$. Let $S_1,S_2,\ldots,S_t$ be nonempty subsets of $[n]$. The subposet induced by the elements of $S_1\times\cdots\times S_t$ is called a \emph{subgrid} of $\nn^t$. Of course, every subgrid is a subposet that is a grid, but the converse is not true.

We will use the Product Ramsey Theorem, which can be phrased with our terminology as follows.

\begin{theorem}[Product Ramsey Theorem \cite{GrahamRothschildSpencer}]
For all $t$, $r$, $m$, and $l$ there exists an $n$ such that for all $r$-colorings of the $\mm^t$ subgrids of $\nn^t$, there is a monochromatic $\ll^t$ subgrid $L$. That is, every $\mm^t$ subgrid of $L$ receives the same color.
\end{theorem}

A relatively easy consequence is the following proposition.

\begin{proposition}\label{gridramsey}
For each $t$ positive integer, the class of $t$-dimensional grids have the Ramsey Property.
\end{proposition}

Before we present a proof, let us recall some classical results of poset theory.
An $\ss^t$ grid $P$ has Dushnik--Miller dimension at most $t$. As such, it has a realizer $\{L_1,\ldots,L_t\}$, which can be used to embed $P$ into $\mathbb{N}^t$: indeed, the $i$th coordinate of the element $x$ can be chosen to be the position of $x$ in $L_i$ (that is, the size of the closed downset of $x$ in $L_i$). In fact this is an embedding into $\ll^t$, where $l=s^t$. It also has the property that for each pair of integers $i,j$ with $1\leq i\leq t$ and $1\leq j\leq l$, there is exactly one $x\in P$ such that the $i$th coordinate of $x$ in the embedding is $j$.

Such embeddings will be called ``casual''. Here is the precise definition. Let $t,s$ be positive integers, and let $l=s^t$. Let $P=\ss^t$, and $Q=\ll^t$. We define the usual projection functions: for $x\in Q$ and $1\leq i\leq t$, the positive integer $\Proj_i(x)$ is the position (size of closed downset) of the $i$th coordinate of $x$ in $\ll$. An embedding $f:P\to Q$ is called \emph{casual}, if for all $i,j$ with $1\leq i\leq t$ and $1\leq j\leq l$, there is exactly one $x\in P$ such that $\Proj_i(f(x))=j$.

So in a casual embedding there are no ties in coordinates. Note that we also require that there are no ``unused'' coordinates. So a casual embedding of $\ss^t$ into a grid $\ll^t$ always has the property that $l=s^t$.

The existence of a casual embedding is typically proven non-constructively, though it is not difficult to construct one.

Now we are ready to prove Proposition~\ref{gridramsey}.

\begin{proof}
Let $t,r$ be positive integers. Let $s$ be a positive integer, and $P$ be an $\ss^t$ grid. We will show that a Ramsey poset $Q$ exists for $P$.

If $s=1$, the theorem is trivial, so we assume $s\geq 2$.

We invoke the Product Ramsey Theorem for $t$ and $r$ as fixed above, for $m=2$, and $l=s^t$, to get a number $n$. We claim that $Q=\nn^t$ is a Ramsey poset for $P$.

To show this, we consider a coloring $c:C(Q)\to[r]$ of the comparabilities of $Q$; here $C(Q)$ denotes the set $\{(a,b)\in Q: \text{$a$ and $b$ are comparable}\}$. We will use this to define an $r$-coloring of the $\mathbf{2}^t$ subgrids of $Q$ as follows. Let $M$ be a $\mathbf{2}^t$ subgrid, with the least element $a$, and the greatest element $b$. Then we assign the color $c(a,b)$ to this subgrid.

By the Product Ramsey Theorem, a monochromatic $\ll^t$ subgrid exists; let this be called $R$. Let $P'$ be a casually embedded copy of $P$ into $R$; we claim $P'$ is monochromatic. To see this, let $a<b$ in $P'$. Since $a$ and $b$ have distinct $i$th coordinates for each $i=1,\ldots,t$ in $R$ (and $Q$), they determine an $M(a,b)$ $\mathbf{2}^t$ subgrid of $R$ (and $Q$). Due to the choice of $R$ by the Product Ramsey Theorem,  each $M(a,b)$ has the same color $r_0$, which, in turn, implies $c(a,b)=r_0$.
\end{proof}

We would like to note that Ramsey's classical theorem is a special case of Proposition~\ref{gridramsey} when $t=1$.

The special case of the theorem of \nesetril\ and R\"odl now follows easily.

\begin{corollary}
The class of all posets has the Ramsey Property.
\end{corollary}

\begin{proof}
Let $P$ be a poset. It is well-know that every poset is a subposet of a large enough Boolean lattice. The Boolean lattice of dimension $d$ is the grid $\mathbf{2}^d$.

So first find a Boolean lattice $B$ such that $P$ is a subposet of $B$. Then use Proposition~\ref{gridramsey} to find a grid $Q$, a Ramsey Poset for $B$. A monochromatic subposet $B$ clearly contains a monochromatic $P$, so the theorem follows.
\end{proof}

Furthermore, since every poset of Dushnik--Miller dimension $d$ can be embedded into $\kk^d$ for sufficiently large $k$, the following corollary is immediate.

\begin{corollary}
The class of posets of dimension at most $d$ has the Ramsey Property.
\end{corollary}

(Of course the corollary remains true if one replaces ``at most'' with ``exactly''.)

Unfortunately none of the tools used here seem to be capable of grasping the complexities of planar posets, so the truth value of Statement~\ref{sta:np2} remains open.

\section{Ramsey Theory of grid subposets}\label{s:subposet}

During this research, we found a statement that would have powerful consequences. It is not a straight generalization of the Product Ramsey Theorem, but it seems to be more useful in many cases. Although the authors disagree on the truth value, for easier discussion we state it as a conjecture.

\begin{conjecture}\label{con:diamond}
For all $t$, $r$, $m$, and $l$ there exists an $n$ such that for all $r$-colorings of the $\mm^t$ subposets of $\nn^t$, there is a monochromatic $\ll^t$ subposet $L$. That is, every $\mm^t$ subposet of $L$ receives the same color.
\end{conjecture}

Note that the difference between the Product Ramsey Theorem and this conjecture is that this conjecture replaces ``subgrids'' in the Product Ramsey Theorem with the more general ``subposets''.

We were able prove this conjecture for $t=2$.

\begin{theorem}\label{diamond}
For all $r$, $m$, and $l$ there exists an $n$ such that for all $r$-colorings of the $\mm^2$ subposets of $\nn^2$, there is a monochromatic $\ll^2$ subposet $P$. That is, every $\mm^2$ subposet of $P$ receives the same color.
\end{theorem}

We break up the proof into smaller parts. The following lemma is interesting in its own right.

\begin{lemma}\label{lemma:uniquerealizer}
Let $P$ be an $\ss^2$ grid, whose ground set is represented by ordered pairs $(i,j)$, with $0\leq i,j\leq s-1$. Then $P$ has only one realizer with two linear extensions. Namely, one linear extension of this realizer is the lexicographic order on the pairs of $P$, and the other is the colexicographic order (the coordinates are considered right-to-left).
\end{lemma}

\begin{proof}
Define the following two sets of ordered pairs of incomparable elements in $P$.
\begin{align*}
I_1&=\big\{\big((1,0),(0,s-1)\big),\big((2,0),(1,s-1)\big),\ldots,\big((s-1,0),(s-2,s-1)\big)\big\}\\
I_2&=\big\{\big((0,1),(s-1,0)\big),\big((0,2),(s-1,1)\big),\ldots,\big((0,s-1),(s-1,s-2)\big)\big\}
\end{align*}

Now let $(x_1,y_1)\in I_1$ and $(x_2,y_2)\in I_2$ be two arbitrary elements of these sets. With appropriate choices of $i,j\in\{0,\ldots,s-1\}$, we have
\begin{gather*}
x_1=(i+1,0),\qquad y_1=(i,s-1),\\
x_2=(0,j+1),\qquad y_2=(s-1,j).
\end{gather*}

Notice $(x_1,y_1)$ and $(x_2,y_2)$ cannot be reversed at the same time in a linear extension: indeed, they form an alternating cycle, because $x_1\leq y_2$, and $x_2\leq y_1$. So every pair in $I_1$ must be reversed in a single linear extension, and the same is true for $I_2$. There is only one linear extension that reverses every pair in $I_1$, and there is only one for $I_2$.

The linear extension that reverses all of $I_1$ is the lexicographic order of the pairs, and the one that reverses $I_2$ is the colexicographic order.
\end{proof}

The key observation to the proof of Theorem~\ref{diamond} is the following.

\begin{lemma}\label{lemma:uniqueembedding}
Let $P$ be an $\ss^2$ grid, and let $Q$ be an $\ll^2$ grid with $l=s^2$. There is (up to automorphism) only one casual embedding of $P$ into $Q$. That is, if $f,g:P\to Q$ are two casual embeddings, then $f(P)=g(P)$, and $f\circ g^{-1}$ (and $g\circ f^{-1}$) are automorphisms of $P$.
\end{lemma}

\begin{proof}
We can think of $Q$ as consisting of pairs $(i,j)$ with $0\leq i,j\leq l-1$ with the natural order. We will use the usual projection functions $\Proj_1((i,j))=i$, and $\Proj_2((i,j))=j$.

Let $f$ be a casual embedding of $P$ into $Q$.
Let $k\in\{1,2\}$, and $h_k=\Proj_k\circ f$. The definition of a casual embedding exactly means that $h_k$ is a bijection from $P$ to $\{0,\ldots,l-1\}$. Now consider the following linear order of the elements of $P$.
\[
L_k=(h_k^{-1}(0),h_k^{-1}(1),\ldots,h_k^{-1}(l-1))
\]
First note that $L_k$ is a linear extension of $P$, because $f$ is an embedding. Due to the same reason, any pair of incomparable elements will be ordered opposite in $L_1$ and $L_2$, so $\{L_1,L_2\}$ is a realizer of $P$. It is also important to recall that the ordered pair of linear extensions $(L_1,L_2)$ uniquely determines the casual embedding $f$. (See discussion after the statement of Proposition~\ref{gridramsey}.)

Now let $g$ be another casual embedding of $P$ into $Q$, and let $M_1,M_2$ be the linear extensions determined by $g$, similarly as above. Since $\{M_1,M_2\}$ is also a realizer of $P$, Lemma~\ref{lemma:uniquerealizer} implies that either $L_1=M_1$ and $L_2=M_2$, or $L_1=M_2$ and $L_2=M_1$.

The former case is simple: since $(L_1,L_2)$ uniquely determines $f$, and $(M_1,M_2)$ uniquely determines $g$, the fact $(L_1,L_2)=(M_1,M_2)$ implies that $f=g$.

So now assume that $L_1=M_2$, and $L_2=M_1$. By Lemma~\ref{lemma:uniquerealizer}, $L_1$ is either the lexicographic order, or the colexicographic order of $P$. By possibly swapping the roles of $f$ and $g$, we may assume it is the former. With these assumptions, $f$ and $g$ are completely determined. It is easy to see that
\[
f((i,j))=(si+j,sj+i)
\qquad\text{and}\qquad
g((i,j))=(sj+i,si+j).
\]

If $(a,b)\in f(P)$, say $(a,b)=f((i,j))$, then $(a,b)=g((j,i))$, so $(a,b)\in g(P)$, and the converse follows the same way. This shows $f(P)=g(P)$. The last part of the statement follows from the fact that composition of isomorphisms is an isomorphism.
\end{proof}

If the conditions of Lemma~\ref{lemma:uniqueembedding} are satisfied, then the subposet induced by the image $f(P)$ under a casual embedding of $P$ into $Q$ will be called the \emph{core} of $Q$. Lemma~\ref{lemma:uniqueembedding} shows that the core of $Q$ is uniquely determined by $Q$, so the usage of the definite article is justified.

\begin{proof}[Proof of Theorem~\ref{diamond}]
Let $M=m^2$, and $L=l^2$. By the Product Ramsey Theorem, there exists a positive integer $n$ such that for all $r$-colorings of the $\MM^2$ subgrids of $\nn^2$, there is a monochromatic $\LL^2$ subgrid. We claim that $n$ satisfies the requirements of our theorem.

To see this, let $c_1$ be an $r$-coloring of the $\mm^2$ subposets of $\nn^2$. We will define an $r$-coloring $c_2$ on the $\MM^2$ subgrids of $\nn^2$. For each $Q\cong\MM^2$ subgrid, let $P$ be the core of $Q$. Let $c_2(Q)=c_1(P)$.

As noted earlier, the $\nn^2$ grid has a monochromatic $\LL^2$ subgrid under the coloring $c_2$; call this $Q'$. Here, monochromatic means that there exists a color $r_0$, such that for every $\MM^2$ subgrid $G$ of $Q'$, we have $c_2(G)=r_0$. Let $P'$ be the core of $Q'$ (see Figure~\ref{fig:dcd}).

\begin{figure}
\begin{center}
\ifbw
\tikzset{core/.style={dashed},large/.style={ultra thick}}
\else
\tikzset{core/.style={red},large/.style={ultra thick,blue}}
\fi
\begin{tikzpicture}
[scale=0.3]
\foreach \i in {0,...,8} {
  \draw (-\i,\i) -- (8-\i, 8+\i);
  \draw (\i,\i) -- (\i-8,\i+8);
  }

\foreach \i in {0,...,2} {
  \draw[core] (-2*\i,4*\i) -- (4-2*\i,8+4*\i);
  \draw[core] (2*\i,4*\i) -- (2*\i-4,4*\i+8);
  }

\draw[large] (0,0) -- (8,8);
\draw[large] (-2,2) -- (6,10);
\draw[large] (-3,3) -- (5,11);
\draw[large] (-8,8) -- (0,16);
\draw[large] (0,0) -- (-8,8);
\draw[large] (1,1) -- (-7,9);
\draw[large] (6,6) -- (-2,14);
\draw[large] (8,8) -- (0,16);

\fill (0,0) circle [radius = 0.3];
\fill (-2,4) circle [radius = 0.3];
\fill (4,8) circle [radius = 0.3];
\fill (0,16) circle [radius = 0.3];
\end{tikzpicture}
\caption{\label{fig:dcd}For the proof of Theorem~\ref{diamond}. In this figure, $m=2$, $l=3$. The $9\times 9$ grid is $Q'$, the core $P'$ is \ifbw dashed. \else red. \fi The points of the subposet $D$ are marked by black dots. The \ifbw thick \else blue \fi grid is $S$.}
\end{center}
\end{figure}
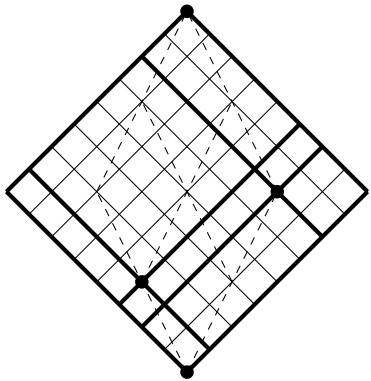

Clearly, $P'\cong\ll^2$. It remains to be seen that every $\mm^2$ subposet of $P'$ received the same color under $c_1$.

Let $D$ be an arbitrary $\mm^2$ subposet in $P'$. Let
\[
S_1=\{\Proj_1(x):x\in D\}\qquad S_2=\{\Proj_2(x):x\in D\}.
\]
where $\Proj_i(x)$ is the $i$th coordinate of $x$ in $\nn^2$. Let $S=S_1\times S_2$. Since $P'$ is a casually embedded copy, $|S_i|=|D|=m^2=M$, so $S\cong \MM^2$, a subgrid of $Q'$. Therefore $c_2(S)=r_0$. On the other hand, $D$ is the core of $S$, so $c_1(D)=r_0$, which finishes the proof.
\end{proof}

Clearly, the techniques used here heavily rely on the fact that $t=2$. E.g.\ Lemma~\ref{lemma:uniquerealizer} and Lemma~\ref{lemma:uniqueembedding} are not true for $t>3$. But the conjecture may still be correct.

Trotter \cite{trotter-personal} suggested that the conjecture is false for $t=3$, offering the following idea for a counterexample.
Let $t=3$, $r=2$, $m=2$, and $l=8$. Suppose the conjecture is true, and there is an appropriate $n$. Now color the $\mm^t=\mathbf{2}^3$ subposets of $\nn^3$ as follows. For a subposet $P$, consider the coordinates of the points in the $\nn^3$. If there is no tie, then the coordinates define a realizer of $P$. There are multiple fundamentally different realizers of $\mathbf{2}^3$, so color $P$ based on what kind of realizer its embedding defines. It does not matter how we group the realizer types into the two color groups, only that there exist two fundamentally different realizers that are colored different. If there is a tie in the coordinates, color $P$ arbitrarily.

Now suppose we have a monochromatic $\ll^3=\mathbf{8}^3$. Within this $\mathbf{8}^3$, we can find both types of realizers, so it cannot be monochromatic.

However, it is not clear that this counterexample works. We conflate the
embeddings into the $\mathbf{8}^3$ with the embedding into the $\nn^3$. For any given $\mathbf{2}^3$, these two embeddings could be fundamentally different in terms of the realizers they generate. It may be possible to find a weird $\mathbf{8}^3$ such that every $\mathbf{2}^3$ subposet of that one has the same type of realizer generated, when we consider their embedding into the $\nn^3$. Indeed, this is another interesting Ramsey-theoretical question.

\section{Matching linear extensions}\label{s:extension}

As evidenced in Section~\ref{s:subposet}, it is interesting to consider linear extensions of posets, and how they behave in Ramsey-theoretical questions.

\begin{theorem}\label{theorem:ptw}
Let $X$ be a poset and let $M$ be a linear extension of $X$. Furthermore, let $k$ be a positive integer. Then there exists a grid $Y\cong \nn^t$ such that for all $L_1,L_2,\ldots,L_k$ linear extensions of $Y$, there is a subposet $X'$ of $Y$ such that
\begin{itemize}
\item $X'\cong X$, evidenced by the embedding $f:X\to Y$;
\item for all $i=1,\ldots,k$, and for all $a,b\in X$, we have $a<b$ in $M$ if and only if $f(a)<f(b)$ in $L_i$.
\end{itemize}
\end{theorem}

Loosely speaking, for any poset $X$ and its linear extension $M$, one can find a large enough grid, so that no matter how we pick a fixed number of linear extensions of that grid, it has an $X$-subposet on which each linear extension conforms with $M$.

A special case of this theorem for $k=1$ was proven by Paoli, Trotter and Walker \cite{PTW}. The way the proof is written in \cite{PTW} contains an error: they attempt to use the infinite version of the Product Ramsey Theorem, which is false. However, the error is easily correctable by just using the finite version, and choosing appropriately large numbers. Our arguments follow their ideas with the necessary correction and generalizations.
One can also prove this result using results of R\"odl and Arman \cite{RodlArman}, but we believe that the proof provided here is more insightful.

We will need the following classical theorem by Rothschild \cite{Rothschild} about partitions. To state this theorem, we will call a partition of a set into $t$ parts, a \emph{$t$-partition}.

\begin{theorem}\label{partitionramsey}
Let $s\leq t$ be positive integers, and $r$ a positive integer. Then there exists a  positive integer $k_0$ such that for all $k\geq k_0$, no matter how one colors the $s$-partitions of $[k]$ with $r$ colors, there exists a monochromatic $t$-partition in the following sense: any $s$-partition generated from that $t$-partition by unifying parts will have the same color.
\end{theorem}

Now we are ready to prove Theorem~\ref{theorem:ptw}.

\begin{proof}[Proof of Theorem~\ref{theorem:ptw}]
To start the proof, we pick $X$ and $M$, although we will not use them at all in the first part of the proof. Let $s=\dim(X)$. We may assume that $s\geq 3$, for otherwise $X$ can be embedded into a $3$-dimensional poset, and apply the theorem to that.

We need to show that for large enough $n$ and large enough $t$, the poset $\nn^t$ has the prescribed property. We will determine the exact value of $n$ and $t$ later. For now just let $Y=[n]^t$ for undetermined, but large $n$ and $t$. Then let $L_1,\ldots,L_k$ be linear extensions of $Y$.

In the next steps, we will apply the Product Ramsey Theorem repeatedly to cut down $Y$. To do this, we will color the $\mathbf{2}^t$ subgrids (referred to as \emph{hypercubes}) of $Y$ with $2^k$ colors in each step.

Let $H$ be a hypercube of $Y$. Then $H=C_1\times\cdots\times C_t$, where $C_i=\{a_i,b_i\}$, and $a_i<b_i$. Every point of $H$ is of the form $(c_1,\ldots,c_t)$, where $c_i=a_i$ or $c_i=b_i$. Once we fix an $H$ hypercube in $Y$, we can identify the points with $0$--$1$ strings (bit strings) of length $t$: we write $0$ if $c_i=a_i$, and we write $1$, if $c_i=b_i$.

We call two incomparable points \emph{antipodal}, if they differ in every bit. So the pair $00\ldots0$, $11\ldots1$, is \emph{not} antipodal, but every other pair with differing bits is. We can call the bit strings corresponding to these pairs of antipodal points, antipodal bit strings. There are $2^{t-1}-1$ pairs of antipodal bit strings.

It will be important later that antipodal bit strings bijectively correspond to $2$-partitions of $[t]$. Indeed, for $i\in[t]$, we can place $i$ into the first or second part based on the $i$th bit.

Enumerate every pair of antipodal bit strings one by one. In each step, we will define a coloring of the hypercubes of Y with $2^k$ colors, then use the Product Ramsey Theorem to cut down $Y$ to a smaller grid.

Let $A$ be the current antipodal pair of bit strings. We define a coloring of the hypercubes of $Y$ as follows. Let $H$ be a hypercube. Recall that $A$ identifies a pair of antipodal points $A_H$ in $H$. For $i=1,\ldots,k$, write `G' (good), if $L_i$ orders the points of $A_H$ as it would be natural by the $i$th coordinate of the corresponding bit strings, write `B' (bad) otherwise. We will have constructed a string of length $k$ consisting of G's and B's. This is the ``color'' of $H$ for the antipodal pair $A$.

As an example: suppose the current antipodal pair is 001101, 110010, and the color of the hypercube $H$ is GBG. The antipodal pair 001101, 110010 determines a pair of points $a$, $b$ of $H$, respectively. The color GBG means that $a<b$ in $L_1$, and $b<a$ in $L_2$, and $L_3$. In this example, $t=6$ and $k=3$.

The careful reader may get worried about the case when $k>t$. However this is not a concern. We will see later that we can always choose a larger $t$, so we can ensure that $t\geq k$.

Now we have defined a coloring of the hypercubes of $H$ for $A$. By the Product Ramsey Theorem, as long as $n$ is large enough, there is a monochromatic grid, as large as we prescribe it. For now let us just prescribe a very large grid, and we will determine that size later.

We will replace $Y$ with this monochromatic grid, and we note the color of it. We will assign this color to the 2-partition that corresponds to the antipodal pair $A$. Then we move on to the next antipodal pair, do the coloring of the hypercubes of (the reduced) $Y$, apply the Product Ramsey Theorem, and produce a large monochromatic subgrid. Reduce $Y$ again to this, and move on.

After going through every antipodal pair, we arrive at a final grid $Y$. This has the property that every hypercube in it is uniform with respect to the order of their antipodal points in the linear extensions $L_i$. We also colored every $2$-partition of $[t]$ with $2^k$ colors.


Now we apply Theorem~\ref{partitionramsey} to get $t_0$ such that if $t\geq t_0$, and if we $2^k$-color the $2$-partitions of $[t]$, then we can find a monochromatic $(s+k)$-partition. (Recall that $s=\dim(X)$.) The usage of $t$ here is no accident: indeed, our original $t$ was to be determined this way. Note that $t$ only depends on $k$ and $s$.

We do have a $2^k$ coloring defined on the $2$-partitions of $[t]$, so now we determine the $(s+k)$-partition $\psi$, whose existence was guaranteed above. Recall that no matter how we unify parts in $\psi$ to get a $2$-partition, it will always have the same color $r_0$, which is a string of G's and B's of length $k$. Somewhat magically, it turns out that we can guarantee that $r_0$ is GG\dots G.

To see this, suppose that the $i$th digit of $r_0$ is B. Let $A\in\psi$ be the part for which $i\in A$, and let $B$ and $C$ be two other parts (recall $s\geq 3$, so $s+k\geq 3$). The partitions $\{B,[t]\setminus B\}$ and $\{C,[t]\setminus C\}$ are both colored $r_0$. Let the corresponding antipodal pairs of bit strings be $b$--$b'$, and $c$--$c'$: let $b$ be the bit string that has $1$'s for indices in $B$, and $0$'s for the rest, and $c$ be the bit string that has $1$'s for indices in $C$, and $0$'s for the rest.

Let $H$ be a hypercube in the final, uniformized $Y$. Carry over the notation $b$, $b'$, $c$, $c'$ to denote the points of $H$ corresponding to these bit strings. Then $b\| b'$, $c\|c'$, $b<c'$, and $c<b'$. In other words, $\{(b,b'),(c,c')\}$ is an alternating cycle. Yet, in $L_i$, the order of these pairs are ``bad'', and by the choice $i\in A$, the $i$th digit of the bit string $b$ is $0$, as well as the $i$th digit of the bit string $c$. So in $L_i$, we have $b>b'$, and $c>c'$, a contradiction. In other words, every digit of the color $r_0$ must be G.


We will use the monochromatic, and all-good partition $\psi$ to embed $X$ into $Y$. The parts of $\psi$ are going to be groups of coordinates that are handled together. To do this, we will need $Y$ to be large enough to accommodate the embedding. So it is time to determine $n$. We must have chosen $n$ to be large enough, so that after $2^{t-1}-1$ repeated applications of the Product Ramsey Theorem, the remaining $Y$ is isomorphic to $[n_0]^t$ with $n_0\geq |X|$. Note that $n$ only depends on $k$ and $t$, and since $t$ itself only depends on $k$ and $s$, at the end $n$ also only depends on $k$ and $s$.

Let $M_1=M_2=\cdots=M_k=M$, and let $\{M_{k+1},\ldots,M_{k+s}\}$ be a realizer of $X$. Clearly, each $M_i$ is a linear extension of $X$, and $\cap M_i=X$. Also, label the parts of $\psi$ with $A_1,\ldots,A_{k+s}$ in such a way that $[k]\subseteq A_1\cup\cdots\cup A_k$. For each $x\in X$, find the position of $x$ in $M_i$ (from below). Let this number be $h_i$ (for \emph{height}). Then we map $x$ to the element $(\chi_1,\ldots,\chi_t)$, where $\chi_j=h_l$ if $j\in A_l$. When we specifically want to emphasize the coordinates of the element $x$, we will write $\chi_j(x)$. Clearly, $1\leq\chi_j(x)\leq|X|$, so this is a mapping from $X$ to $Y$. As we promised, coordinates with indices in the same part of $\psi$ are grouped together so that they all get same value. It is also clear that this is an embedding.

It remains to be seen that for each $i$, $L_i$ conforms $M$. Fix $i\in[k]$, and let $a\|b$ be elements of $X$ so that $a<b$ in $M$. Consider the hypercube
\[
H=\{\chi_1(a),\chi_1(b)\}\times\cdots\times\{\chi_t(a),\chi_t(b)\},
\]
and the bit string
\[
d=[\chi_1(a)>\chi_1(b)]\ldots[\chi_t(a)>\chi_t(b)].
\]
(Here we used the Iverson bracket notation.)

We know that at least the first $k$ digits of $d$ are $0$, in particular the $i$th digit is $0$. Since the antipodal points corresponding to $d$ and $d'$ (the complement of $d$) in $H$ are $a$ and $b$, respectively, and since $H$ was colored GG\ldots G, it shows that $a<b$ in each of $L_1,\ldots,L_k$. This finishes the proof.
\end{proof}

It may be tempting to attempt to generalize this theorem further. After all, it may seem that if we have $k$ linear extensions of $X$, say $M_1,\ldots,M_k$, most of the proof still goes through. One may think that if we perform the embedding at the end carefully, we could make $L_i$ conform with $M_i$ for all $i=1,\ldots,k$.

This, however, is not the case, and there is a very simple counterexample. Just choose any poset $X$ that has at least two fundamentally different linear extensions $M_1,M_2$; we will make $k=2$. Then, if an appropriate $Y$ exists, we pick $L_1=L_2$. Clearly, we cannot have both $L_1$ conforming with $M_1$, and $L_2$ conforming with $M_2$.

\section{Acknowledgment}

We express our gratitude toward the anonymous referees for the careful reading and valuable suggestions on how to improve this paper.


\begin{thebibliography}{10}

\bibitem{planar}
{\sc Axenovich, M., Schade, U., Thomassen, C., and Ueckerdt, T.}
\newblock Planar {R}amsey graphs.
\newblock {\em Electron. J. Combin. 26}, 4 (2019), Paper No. 4.9.

\bibitem{dimensions}
{\sc Barrera-Cruz, F., Prag, T., Smith, H.~C., Taylor, L., and Trotter, W.~T.}
\newblock Comparing {D}ushnik-{M}iller dimension, {B}oolean dimension and local
  dimension.
\newblock {\em Order 37}, 2 (2020), 243--269.

\bibitem{local}
{\sc Bosek, B., Grytczuk, J., and Trotter, W.~T.}
\newblock Local dimension is unbounded for planar posets.
\newblock {\em Electron. J. Combin. 27}, 4 (2020), Paper No. 4.28, 12.

\bibitem{Deuber}
{\sc Deuber, W.}
\newblock Generalizations of {R}amsey's theorem.
\newblock In {\em Infinite and finite sets ({C}olloq., {K}eszthely, 1973;
  dedicated to {P}. {E}rd\H{o}s on his 60th birthday), {V}ol. {I}}. 1975,
  pp.~323--332. Colloq. Math. Soc. J\'{a}nos Bolyai, Vol. 10.

\bibitem{Diestel}
{\sc Diestel, R.}
\newblock {\em Graph theory}, fifth~ed., vol.~173 of {\em Graduate Texts in
  Mathematics}.
\newblock Springer, Berlin, 2018.

\bibitem{ErdosHajnalPosa}
{\sc Erd\H{o}s, P., Hajnal, A., and P\'{o}sa, L.}
\newblock Strong embeddings of graphs into colored graphs.
\newblock In {\em Infinite and finite sets ({C}olloq., {K}eszthely, 1973;
  dedicated to {P}. {E}rd\H{o}s on his 60th birthday), {V}ol. {I}}. 1975,
  pp.~585--595. Colloq. Math. Soc. J\'{a}nos Bolyai, Vol. 10.

\bibitem{FelsnerMeszarosMicek}
{\sc Felsner, S., M\'{e}sz\'{a}ros, T., and Micek, P.}
\newblock Boolean dimension and tree-width.
\newblock {\em Combinatorica 40}, 5 (2020), 655--677.

\bibitem{Goncalves}
{\sc Gon\c{c}alves, D.}
\newblock Edge partition of planar graphs into two outerplanar graphs.
\newblock In {\em S{TOC}'05: {P}roceedings of the 37th {A}nnual {ACM}
  {S}ymposium on {T}heory of {C}omputing\/} (2005), ACM, New York,
  pp.~504--512.

\bibitem{GrahamRothschildSpencer}
{\sc Graham, R.~L., Rothschild, B.~L., and Spencer, J.~H.}
\newblock {\em Ramsey theory}.
\newblock Wiley Series in Discrete Mathematics and Optimization. John Wiley \&
  Sons, Inc., Hoboken, NJ, 2013.

\bibitem{MeszarosMicekTrotter}
{\sc M\'{e}sz\'{a}ros, T., Micek, P., and Trotter, W.~T.}
\newblock Boolean dimension, components and blocks.
\newblock {\em Order 37}, 2 (2020), 287--298.

\bibitem{NesetrilPudlak}
{\sc Ne\v{s}et\v{r}il, J., and Pudl\'{a}k, P.}
\newblock A note on {B}oolean dimension of posets.
\newblock In {\em Irregularities of partitions ({F}ert\H{o}d, 1986)}, vol.~8 of
  {\em Algorithms Combin. Study Res. Texts}. Springer, Berlin, 1989,
  pp.~137--140.

\bibitem{NesetrilRodl}
{\sc Ne\v{s}et\v{r}il, J., and R\"{o}dl, V.}
\newblock Combinatorial partitions of finite posets and lattices---{R}amsey
  lattices.
\newblock {\em Algebra Universalis 19}, 1 (1984), 106--119.

\bibitem{PTW}
{\sc Paoli, M., Trotter, W.~T., and Walker, J.~W.}
\newblock Graphs and orders in {R}amsey theory and in dimension theory.
\newblock {\em Graphs and Orders\/} (1985), 351--394.

\bibitem{Rodl-MT}
{\sc R\"{o}dl, V.}
\newblock Generalization of {R}amsey theorem and dimension of graphs.
\newblock Master's thesis, Charles University, Prague, 1973.

\bibitem{RodlArman}
{\sc R\"{o}dl, V., and Arman, A.}
\newblock Note on a {R}amsey theorem for posets with linear extensions.
\newblock {\em Electron. J. Combin. 24}, 4 (2017), Paper No. 4.36.

\bibitem{Rothschild}
{\sc Rothschild, B.~L.}
\newblock {\em A generalization of {R}amsey's Theorem and a conjecture of
  {R}ota}.
\newblock PhD thesis, Yale University, 1967.

\bibitem{trotter-personal}
{\sc Trotter, W.~T.}
\newblock Personal communication.

\bibitem{Trotter}
{\sc Trotter, W.~T.}
\newblock {\em Combinatorics and partially ordered sets}.
\newblock Johns Hopkins Series in the Mathematical Sciences. Johns Hopkins
  University Press, Baltimore, MD, 1992.
\newblock Dimension theory.

\end{thebibliography}
\end{document}